\documentclass[a4paper,11pt]{amsart}
\usepackage[utf8]{inputenc}

\usepackage{amsmath}
\usepackage{amsthm}
\usepackage{amssymb}
\usepackage{enumerate}
\usepackage{setspace}
\usepackage{caption}
\usepackage{mathrsfs}
\usepackage{graphicx}
\usepackage{hyperref}
\usepackage{esint}
\usepackage{comment}
\usepackage[normalem]{ulem}
\usepackage[all]{xy}

\usepackage{color}
\usepackage{lipsum}

\newtheorem{theorem}{Theorem}[section]
\newtheorem{lemma}[theorem]{Lemma}
\newtheorem{corollary}[theorem]{Corollary}

\newtheorem{question}[theorem]{Question}
\numberwithin{equation}{section}

\theoremstyle {definition}

\newtheorem{remark}[theorem]{Remark}

\DeclareMathOperator{\scal}{scal}

\DeclareMathOperator{\vol}{vol}

\DeclareMathOperator{\diam}{diam}

\DeclareMathOperator{\conj}{conj}

\DeclareMathOperator{\inj}{inj}

\DeclareMathOperator{\id}{id}
\newcommand{\eps}{\varepsilon}

\title[Strenghened injectivity radius bounds for psc
manifolds]{Strengthened injectivity radius bounds for manifolds with
  positive scalar curvature}
\author{Thomas Richard}
\address{LAMA, Univ Paris Est Creteil, Univ Gustave Eiffel, CNRS, F-94010, Créteil, France}
\email{thomas.richard@u-pec.fr}

\date{}

\begin{document}

\maketitle
\begin{abstract}
    Green's inequality shows that a compact
    Riemannian manifold with scalar curvature at least $n(n-1)$ has
    injectivity radius at most $\pi$, and that equality is achieved
    only for the radius 1 sphere. In this work we show how extra
    topological assumptions can lead to stronger upper bounds. The
    topologies we consider are
    $\mathbb{S}^2\times\mathbb{T}^{n-k-2}\times\mathbb{R}^k$ for
    $n\leq 7$ and $0\leq k\leq 2$ and 3-manifolds with positive scalar
    curvature except lens spaces $L(p,q)$ with $p$ odd. We also prove
    a strengthened inequality for $3$-manifolds with positive scalar
    curvature and large diameter. Our proof uses
    previous results of Gromov and Zhu.
    
\end{abstract}

In stark contrast with positive lower bounds on the sectional or Ricci
curvature, positive scalar curvature doesn't provide any control on
the volume or diameter of a Riemannian manifold $(M^n,g)$ of dimension
at least 3 as is already clear by considering product metrics on $\mathbb{S}^2\times\mathbb{T}^1$. However in 1963, Leon Green proved the following result:
\begin{theorem}\cite{GreenWieder}
    Let $(M^n,g)$ be a complete manifold with finite volume and Ricci
    curvature bounded from below. Assume \[\bar
      s=\frac{1}{\vol(M^n,g)}\int_M\scal_gdv_g>0.\]
    Then the conjugacy radius of $(M^n,g)$ satisfies
    \[\conj(M^n,g)\leq\pi\sqrt{\frac{n(n-1)}{\bar s}}.\]
    Moreover, if equality is achieved then $(M^n,g)$ has constant sectional curvature metric.
\end{theorem}

Since the injectivity radius is smaller than the conjugate radius and
the round sphere is the only spherical space form with injectivity
radius $\pi$, a straightforward corollary is:
\begin{corollary}
  If a compact manifold $(M^n,g)$ has $\scal_g\geq n(n-1)$, then
  $\inj_g\leq\pi$. Moreover if equality is achieved then $(M^n,g)$ is
  isometric to the round sphere.
\end{corollary}

As with other sharp geometric inequalities, it is natural to ask if some form of stability occurs:
\begin{question}
 If a compact $(M^n,g)$ has $\scal_g\geq n(n-1)$ and injectivity radius close to $\pi$, is it close to the round sphere in some sense ?
\end{question}

A more precise (and less ambitious) version of this question is to look for a purely topological answer:
\begin{question}\label{questtop}
 If a compact $(M^n,g)$ has $\scal_g\geq n(n-1)$ and injectivity
 radius close to $\pi$, is it homeomorphic to the sphere ?
\end{question}
Some progress has been made in this direction by Tuschmann and Wiemeler in \cite{Tuschmann} under additional Ricci curvature and volume bounds.

Another way to look at this problem is to ask to ask for a
strenghening of Green's inequality in a given non trivial topological
context, namely:
\begin{question}
  Given a smooth $n$-manifold $M^n$ which is not homeomorphic to
  $\mathbb{S}^n$ and a metric $g$ on $M^n$ with $\scal_g\geq n(n-1)$,
  can one find $\iota=\iota(M^n)<\pi$ such that $\inj_g\leq \iota$ ?
\end{question}

To the knowledge of the
author, the only known result of this kind
follows from \cite{BBEN}, though the authors didn't phrase it in this
way, where it is shown that if $(\mathbb{RP}^3,g)$
has $\scal_g\geq 6$ then its $1$-systole is at most $\pi$ and hence
its injectivity radius is at most $\pi/2$. Gromov's
\cite{gromov2023scalar} also contains related conjecture and partial results.

To get a better feel of the problem we compile here the scalar
curvature and injectivity radius of some classical examples built from
rank one symmetric spaces and products in Table
\ref{tab:injex}, together with what the injectivity radius would be
after rescaling to have the same scalar curvature as the sphere.
\begin{table}[h]
  \centering
  \begin{tabular}{|l|c|c|c|}
    \hline
    Manifold $(M^n,g)$ & $\scal_g$ & $\inj_g$ &
                                                $\sqrt{\frac{\scal_g}{n(n-1)}}\inj_g$
    \\
    \hline
    $\mathbb{S}^n$ & $n(n-1)$ & $\pi$ & $\pi$\\
    $\mathbb{S}^{n-1}\times\mathbb{T}^1$ & $(n-1)(n-2)$ & $\pi$ & $\sqrt{\frac{n-2}{n}}\pi$\\
    $\mathbb{S}^{n-2}\times\mathbb{S}^2$, $n\geq 4$ & $n^2-5n+10$ & $\pi$ & $\sqrt{\frac{n^2-5n+10}{n(n-1)}}\pi$\\
    $\mathbb{RP}^n$ & $n(n-1)$ & $\pi/2$ & $\frac{\pi}{2}$\\
    $\mathbb{S}^2\times\mathbb{T}^{n-2}$ & $2$ & $\pi$
                                              & $\sqrt{\frac{2}{n(n-1)}}\pi$\\
    $\left(\mathbb{S}^2\right)^{n/2}$, $n$ even & $n$
         & $\pi$ & $\frac{1}{\sqrt{ n-1}}\pi$\\
    $\mathbb{CP}^{n/2}$, $n$ even & 
    $n(n+2)$  & $\pi/2$ & $\sqrt{\frac{n+2}{n-1}}\frac{\pi}{2}$ \\
    $\mathbb{HP}^{n/4}$, $n\equiv 0\ [4]$ &
                                                                 $n(n+8)$
                                   & $\pi/2$ &
                                               $\sqrt{\frac{n+8}{n-1}}\frac{\pi}{2}$ \\
    $\mathbb{OP}^2$, $n=16$ & $ 576$ & $\pi/2$ &
                                               $\sqrt{\frac{12}{5}}\frac{\pi}{2}$\\
    \hline
  \end{tabular}
  \caption{Injectivity radius of some classical examples of $n$-manifolds.}
  \label{tab:injex}
\end{table}

Of all the non spheres examples we have tested,
$\mathbb{S}^{n-1}\times\mathbb{T}^1$ always achieves the biggest injectivity
radius for a given scalar curvature bound for a given $n$, though in
dimension 4 it is tied with the Fubini-Study metric on $\mathbb{CP}^2$. For instance in dimension
$3$, $\mathbb{S}^2\times\mathbb{T}^1$ carries a (product) metric with scalar curvature $6$ and
injectivity radius $\pi/\sqrt{3}$.

We here prove some results in this direction.

In dimension 3 we get:
\begin{theorem}\label{thm3}
    Let $(M^3,g)$ be a compact $3$-manifold with $\scal_g\geq 6$ and $\inj_g>\frac{2\pi}{3}$, then $M^3$ is diffeomorphic to a lens space $L(p,q)$ with $p$ odd.
\end{theorem}
Recall that the lens space $L(p,q)$ (defined for coprime integers $p$ and $q$) is the quotient of $\mathbb{S}^3\subset\mathbb{C}^2$ by the action of $\mathbb{Z}/p\mathbb{Z}$ generated by the diffeomorphism: 
\begin{equation}\label{action}
    (z,w)\mapsto\left(e^{i\tfrac{2\pi}{p}}z,e^{i\tfrac{2\pi q}{p}}w \right).
\end{equation}
The occurence of the odd lens spaces in the conclusion of the theorem
is probably not necessary and one could hope to conclude that $M^n$ is
actually diffeomorphic to $\mathbb{S}^3$.

Compactness of $M$ is actually not necessary since we will also show:
\begin{theorem}\label{thm3comp}
    Let $(M^3,g)$ be a complete non-compact $3$-manifold with $\scal_g\geq 6$, then $\inj_g\leq\frac{2\pi}{3}$.
\end{theorem}
Note that since no volume assumption is made, Green's upper bound of
$\pi$ does not apply here. However Gromov indicates in
\cite{gromov2023scalar} that a (possibly huge) upper bound can be
derived for complete $n$-manifolds with uniformly positive scalar
curvature with $n\leq 5$ by quantifying the contractibility argument
in \cite{ChodoshLi}.

The proof of Theorem \ref{thm3comp} can be adapted to yield an
improvement of Green's  bound for manifolds with big diameter. For convenience we
set:
\begin{align*}
  \bar D: (2\pi/3,\pi) &\to\mathbb{R}\\
  r&\mapsto
2r+\frac{2\pi}{3\sqrt{1-\tfrac{4\pi^2}{9r^2}}}.
\end{align*}
and include a plot of $\bar D $ as Figure \ref{fig:plotofphi}.
\begin{theorem}
     \label{thm3diam}
     
     Let $(M^3,g)$ be a compact Riemannian 3-manifold with $\scal_g\geq 6$ and let $r\in (\tfrac{2\pi}{3},\pi)$. If:
     \[\diam(M^3,g)>\bar D(r),\] then:
     \[\inj_g\leq r.\]
\end{theorem}

\begin{figure}[h]
  \centering
  \includegraphics [scale=1]{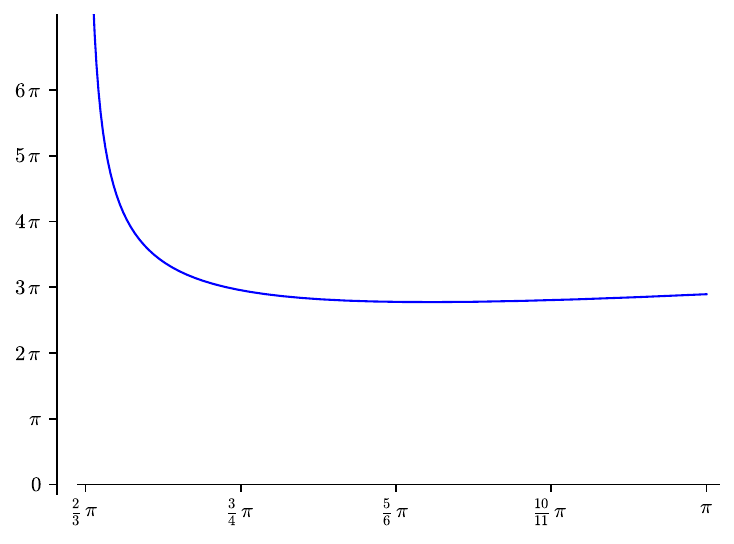}
  \caption{Plot of $\bar D$ on $(2\pi/3,\pi)$.}
  \label{fig:plotofphi}
\end{figure}

Since \[\bar D(r)\xrightarrow{r\to2\pi/3} +\infty,\]
one consequence of this result is that a manifold $(M^3,g)$ with
$\scal_g\geq 6$ and very big diameter cannot have injectivity radius
much bigger than $\frac{2\pi}{3}$.
On the other hand the smallest diameter for which the theorem applies can be found by
minimizing $\bar D$ on $(2\pi/3,\pi)$. Unfortunately
studying the sign of $\bar D '$ leads to a cubic equation
in $r$ whose
solutions are not pretty. However this minimum of $\bar D $ is close to $r=\tfrac{5\pi}{6}$ for which we get
$\bar D (r)=\tfrac{25\pi}{9}$. Hence we have that if $(M^3,g)$ has
$\scal_g\geq 6$ and $\diam(M^3,g)>\tfrac{25\pi}{9}\simeq 2.777\pi$ then
$\inj_g\leq \frac{5\pi}{6}\simeq 0.833\pi$. Numerically $\bar D $
attains its minimum of approximately $2.775\pi$ at $r\simeq 0.851\pi$.

Note that given our previous result, this last theorem is only
interesting if $M^3$ is an odd-order lens space, otherwise Theorem
\ref{thm3} gives a better estimate.

In dimension less than 7 we obtain:
\begin{theorem}\label{thmn}
    Let $3\leq n\leq 7$, $k\in\{0,1,2\}$ and let $M^n$ be a manifold such that there exists a proper map $F:M^n\to \mathbb{S}^2\times\mathbb{T}^{n-k-2}\times\mathbb{R}^k$, of non-vanishing degree. Let $g$ be a complete metric on $M^n$ with $\scal_g\geq n(n-1)$, then $\inj_g\leq \frac{2\pi}{n}$.
\end{theorem}
For $k=0$, this result can be found in \cite [1.1
(b)]{gromov2023scalar} at least if $M^n$ is assumed to be homeomorphic
to a product $\mathbb{S}^2\times\mathbb{T}^{n-2}$, however the proof
is only sketched so we include it here.

When $k$ is not $0$, we say that $F:M^n\to
\mathbb{S}^2\times\mathbb{T}^{n-k-2}\times\mathbb{R}^k$ has
non-vanishing degree if the induced map on compactly supported $n$-dimensional de Rahm cohomology $F^*:H^n_c(\mathbb{S}^2\times\mathbb{T}^{n-k-2}\times\mathbb{R}^k)\to H^n_c(M^n)$ is not zero.

Both proofs use a Bonnet-Myers type inequality by Gromov for
$\mathbb{T}^{n-2}$-invariant positive scalar curvature metrics on
$\mathbb{S}^2\times\mathbb{T}^{n-2}$. In dimension 3, this inequality
is used to give a diameter estimate for stable immersed minimal
spheres. In higher dimension, we use a construction by Zhu
\cite{Zhu}. Our quantitative result (Theorem \ref{thm3diam}) and our
results for complete non compact manifolds (Theorem \ref{thm3comp} and
Theorem \ref{thmn} for $k\neq 0$) rely on the use of Gromov's
$\mu$-bubbles (see Theorem \ref{thm-gro-mub}).

It is at present unclear if this $\frac{2\pi}{n}$ bound is optimal: it comes
from the diameter estimate from \autoref{thmGBM}. One could believe that suitably rescaled product
metrics on $\mathbb{S}^2\times\mathbb{T}^{n-2}$ are optimal: these
have injectivity radius $\sqrt{\frac{2}{n(n-1)}}\pi$. However since \autoref{thmGBM} by itself is optimal as recently shown in \cite{HuXuZhang} different methods need to be used to investigate this question.

%
%


We now outline the rest of the paper. In section \ref{reminders} we
state the symmetrization results by Gromov and Zhu that we will need,
in section \ref{GBM} we spend some time discussing Gromov's toroidal
band inequality and Bonnet-Myers type inequality, in section
\ref{dimn} we prove Theorem \ref{thmn} for $k=0$, in section
\ref{dim3} we prove Theorem \ref{thm3} and explain why it is not obvious
to strengthen Theorem \ref{thm3} by removing the remaining lens spaces
from its conclusion. We then treat the case of
3-manifolds with big diameter in \autoref{dim3diam} and the case of
open 3-manifolds in \autoref{dim3diam}~\autoref{sec:inject-radi-open}. Theorem \ref{thmn} for $k\neq
0$ is treated in \autoref{sec:inject-radi-mathbbs2}. The non
optimality of products for Gromov's \autoref{thmGBM} is proved in \autoref{sec:non-optim-prod}.

\subsubsection*{Acknowledgment}
This work was carried out during the author's stay at PIMS-CNRS IRL in Vancouver, I thank PIMS and UBC's mathematics department for their hospitality. I thank W. Tuschmann for telling me about Question \ref{questtop} and
L. Bessières and A. Fraser for useful discussions. I also thank the referee for his valuable comments.

\section{Gromov's and Zhu's symmetrizations}
\label{reminders}
We state two symmetrization results based on the stable minimal hypersurface method. 
The first is Zhu's construction, which is an elegant application of Fischer-Colbrie--Schoen symmetrization to $\mathbb{S}^2\times\mathbb{T}^{n-2}$.
\begin{theorem}\cite[Prop. 2.2]{Zhu}\label{thmZhu}
     Let $3\leq n\leq 7$ and let $M^n$ such that there exists a
     non-zero degree map
     $F:M^n\to\mathbb{S}^2\times\mathbb{T}^{n-2}$. Let $g$ be a
     riemannian metric on $M^n$ such that $\scal_g\geq n(n-1)$, then
     there exists a genus 0 surface $\Sigma^2\subset M^n$ such that:
     \begin{itemize}
     \item $\int_{\Sigma}F^*\sigma\neq 0$ where $\sigma$ is the
       fundamental cohomology class of
       $\mathbb{S}^2\subset\mathbb{S}^2\times\mathbb{T}^{n-2}$.
       \item there are $n-2$ smooth positive functions $f_1,\dots,f_{n-2}:\Sigma^2\to\mathbb{R}$ such that $\Sigma^2\times\mathbb{T}^{n-2}$ endowed with the metric $\tilde{g}=g_{|\Sigma}+\sum_{i=1}^{n-2}f_i^2d\theta_i^2$ satisfies $\scal_{\tilde{g}}\geq n(n-1)$.
     \end{itemize}
\end{theorem}

$\Sigma$ is such that there exists nested $k$-dimensional submanifolds
$M_k$ for $2\leq k\leq n$ with:
\begin{itemize}
\item $\Sigma^2=M_2\subset M_3\subset\cdots \subset M_n=M$,
\item $M_{k-1}\times\mathbb{T}^{n-k}$ is a stable minimal surface in $(M_k\times\mathbb{T}^{n-k},g_{|M_k}+\sum_{i=1}^{n-k}f_{i,k}^2d\theta_i^2)$ for some positive smooth functions $f_{i,k}:M_k\to\mathbb{R}$.
\item If $\pi_k$ denotes the projection
  $\mathbb{S}^2\times\mathbb{T}^{n-2}\to\mathbb{S}^2\times\mathbb{T}^{k-2}$
  then $F_k=\pi_k\circ F_{|M_k}:M_k\to\mathbb{S}^2\times\mathbb{T}^{k-2}$
  has non zero degree.
  \item $M_{k-1}\subset M_k$ is Poincaré dual to $F_k^*d\theta_k$ where
    $\theta_k$ is a coordinate on the last factor of $\mathbb{S}^2\times\mathbb{T}^{k-2}$
\end{itemize}
Moreover, each metric $g_{|M_k}+\sum_{i=1}^{n-k}f_{i,k}^2d\theta_i^2$ has scalar curvature at least $n(n-1)$.

We now state a symmetrization theorem obtained by Gromov using $\mu$-bubbles for manifolds with positive scalar curvature with two well separated boundary components.
\begin{theorem}\cite[Section 3.7]{Gromov4}\label{thm-gro-mub}
  Let $(M^n,\partial_{\pm},g)$ $(n\leq 7)$ be a Riemannian band. Assume that
  \[\scal_g\geq \frac{4(n-1)\pi^2}{n\, d_g(\partial_-,
      \partial_+)^2}+\delta\]
  for some $\delta>0$. Then there exists
  \begin{itemize}
  \item an hypersurface $\Sigma$ which separates
    $\partial_-$ and $\partial_+$,
  \item a positive function
    $u:\Sigma\to\mathbb{R}$,
  \end{itemize}
  such that the metric
  $h=g|_{\Sigma}+u^2dt^2$ on $\Sigma\times\mathbb{R}$ has $\scal_h\geq \delta$.
\end{theorem}

\section{Gromov's toroidal band inequality and stabilized Bonnet--Myers theorem}
\label{GBM}
We first present Gromov's inequality on positive scalar curvature metrics on $[-1,1]\times\mathbb{T}^{n-1}$ :
\begin{theorem}\label{ineqGro}
    Let $n\geq 2$ and $g$ be a metric on $[-1,1]\times\mathbb{T}^{n-1}$ with $\scal_g\geq n(n-1)$ then:
    \[d_g\left(\{-1\}\times\mathbb{T}^{n-1},\{+1\}\times\mathbb{T}^{n-1}\right)\leq\frac{2\pi}{n}\]
\end{theorem}
Note that if $dx^2$ is a flat metric on $\mathbb{T}^{n-1}$ then the metric $dt^2+\left(\cos\frac{nt}{2}\right)^{4/n}dx^2$ on $(-\tfrac{\pi}{n},\tfrac{\pi}{n})\times\mathbb{T}^{n-1}$ has constant scalar curvature equal to $n(n-1)$ while its boundary components are $\frac{2\pi}{n}$ apart, hence the inequality is optimal.

This inequality was first proved in dimension $n\leq 7$ using the stable hypersurface method and Fischer-Colbrie--Schoen symmetrization in \cite{GromovM}, and was subsequently expanded to higher dimension by Cecchini \cite{Cecchini} using Dirac operators methods. 
It can also be proved using $\mu$-bubbles, which was Gromov's original use of Theorem \ref{thm-gro-mub}.
For an in-depth discussion of this inequality and its various generalizations, see sections 3.6 to 3.8 of \cite{Gromov4}.

We now state (and prove for convenience of the reader) Gromov's Bonnet-Myers theorem on $\mathbb{T}^{n-2}$-invariant positive scalar curvature metrics on $\mathbb{S}^2\times\mathbb{T}^{n-2}$ and make some remarks on its optimality.
\begin{theorem}[{\cite[2.8]{Gromov4}}]\label{thmGBM}
     Let $g$ be a metric on $\mathbb{S}^{2}$ and $f_1,\dots,f_{n-2}:\mathbb{S}^{2}\to\mathbb{R}$ be smooth positive functions such that the metric $\tilde g=g+\sum_{i=1}^{n-2}f_i^2d\theta_i^2$ on $\mathbb{S}^2\times\mathbb{T}^{n-2}$ has $\scal_{\tilde{g}}\geq n(n-1)$. Then:
     \[\diam(\mathbb{S}^2,g)\leq\frac{2\pi}{n}.\]
\end{theorem}
\begin{proof}
    Consider two points $p_-,p_+\in\mathbb{S}^2$ such that $d_g(p_-,p_+)=\diam(\mathbb{S}^2,g)$ and set $B_+=B_g(p_+,\eps)$ and $B_-=B_g(p_-,\eps)$ for some positive $\eps$ less than half the diameter. 
    
    Set $M^n=\mathbb{S}^2\backslash (B_+\cup B_-)\times\mathbb{T}^{n-2}$. Since $\mathbb{S}^2\backslash (B_+\cup B_-)$ is diffeomorphic to $[-1,1]\times\mathbb{T}^1$, $M^n$ is diffeomorphic to $[-1,1]\times\mathbb{T}^{n-1}$, where the boundary component $\{\pm 1\}\times\mathbb{T}^{n-1}$ corresponds to $C_{\pm}\times\mathbb{T}^{n-2}$ (where $C_\pm=\partial B_\pm$). By assumption $(M^n,\tilde g)$ has $\scal_{\tilde g}\geq n(n-1)$ hence Theorem \ref{ineqGro} gives that:
    \[d_{\tilde g}(C_{-}\times\mathbb{T}^{n-2},C_{+}\times\mathbb{T}^{n-2})\leq\frac{2\pi}{n}\]
    
    We have that $d_{\tilde g}(C_{-}\times\mathbb{T}^{n-2},C_{+}\times\mathbb{T}^{n-2})=d_{ g}(C_{-},C_{+})$. To see this consider points $x_\pm\in C_{\pm}\times\mathbb{T}^{n-2}$ and any curve $c:[a,b]\to M^n$ from $x_-$ to $x_+$ and write it as $c=(\gamma,\theta_1,\dots,\theta_{n-2})$ where $\gamma:[a,b]\to \mathbb{S}^2\backslash (B_+\cup B_-)$, then estimate the length of $c$ by:
    \begin{align*}
        \mathcal{L}_{\tilde g}(c)&=\int_a^b\sqrt{g(\dot\gamma,\dot\gamma)+\sum_{i=1}^{n-2}(f_i\circ\gamma)^2\dot\theta_i^2}dt\\
        &\geq \int_a^b\sqrt{g(\dot\gamma,\dot\gamma)}dt=\mathcal{L}_{g}(\gamma)
    \end{align*}
    which shows after taking infimums that \[d_{g}(C_-,C_+)\leq
      d_{\tilde
        g}\left(C_-\times\mathbb{T}^{n-2},C_-\times\mathbb{T}^{n-2}\right).\]
    Equality follows from taking $c$ with $\mathbb{S}^2$ component a minimizing geodesic and constant $\mathbb{T}^{n-2}$ component.
    
    This shows that $\diam(\mathbb{S}^2,g)-2\eps= d_{g}(C_-,C_+)\leq \frac{2\pi}{n}$ for any small enough $\eps$, hence the Theorem is proved.
\end{proof}

\begin{remark} When this paper was written the optimality of the $\tfrac{2\pi}{n}$ in Theorem \ref{thmGBM} was an open problem and previous versions of this article included an inconclusive discussion of this problem with a proof that product metrics could not be optimal given as Appendix \ref{sec:non-optim-prod}. Since then, Hu, Xu and Zhang showed in \cite{HuXuZhang} that the $\frac{2\pi}{n}$ upper bound in Theorem \ref{thmGBM} is indeed optimal. Their proof is worded in the context of a diameter bound for stable minimal surfaces in 3-manifolds with positive scalar curvature (as in Lemma \ref{diamEst}) but the method extends to show optimality in any dimension. These example do not change the conclusions of the present paper as the sequence of examples which show optimality have injectivity radius going to zero. I keep the appendix included because the local non optimality of products doesn't follow from  \cite{HuXuZhang} and the computations could be of some use. 
\end{remark}
\section{Injectivity radius bound for $\mathbb{S}^2\times\mathbb{T}^{n-2}$}
\label{dimn}

We now prove Theorem \ref{thmn} for $k=0$.

\begin{proof}
    Let $M^n$ be an $n$-manifold with a nonvanishing degree map $M^n\to\mathbb{S}^2\times\mathbb{T}^{n-2}$ endowed with a metric $g$ with $\scal_g\geq n(n-1)$. We use Theorem \ref{thmZhu} to get an homologically non-trivial $\Sigma^2\subset M^n$ and positive smooth functions $f_1,f_2,\dots,f_{n-2}$ such that $\Sigma^2\times\mathbb{T}^{n-2}$ endowed with the metric $\tilde{g}=g_{|\Sigma}+\sum_{i=1}^{n-2}f_i^2d\theta_i^2$ satisfies $\scal_{\tilde{g}}\geq n(n-1)$. 
    Note that since $(\Sigma^2,g_{|\Sigma^2})$ is isometrically immersed in $(M^n)$, the extrinsic diameter of $\Sigma^2$ as a subset of $(M^n,g)$ denoted by $\diam_g(\Sigma^2)$ is smaller than its intrinsic diameter $\diam_{g_{|\Sigma^2}}(\Sigma^2)$. In short:
    \begin{equation}\label{ineqdiam}
        \diam_g(\Sigma^2)\leq \diam_{g_{|\Sigma^2}}(\Sigma^2).
    \end{equation}
    
    Applying Theorem \ref{thmGBM} to $(\Sigma^2\times\mathbb{T}^{n-2},\tilde g)$ gives that \begin{equation}\label{diam2pi}
        \diam_{g_{|\Sigma^2}}(\Sigma^2)\leq \tfrac{2\pi}{n}.
    \end{equation}
    Hence by \eqref{ineqdiam} and \eqref{diam2pi} we get:
    \begin{equation}\label{extdiam}
        \diam_{g}(\Sigma^2)\leq \tfrac{2\pi}{n}.
    \end{equation}
    
    Now assume $(M^n,g)$ has injectivity radius bigger than $\tfrac{2\pi}{n}$. Then for any $p\in\Sigma^2$, $B_M(p,\tfrac{2\pi}{n})$ is contractible. The diameter estimate $\eqref{extdiam}$ implies that $\Sigma^2\subset B_M(p,\tfrac{2\pi}{n})$ thus $\Sigma^2$ is trivial in homology which contradicts the definition of $\Sigma^2$.
\end{proof}
\section{Injectivity radius bounds for 3-manifolds with rich topology}
\label{dim3}

Before proving Theorem \ref{thm3}, we first state a diameter estimate for stable minimal immersions of 2-spheres in manifolds with positive scalar curvature initially observed by Schoen and Yau (see Lemma 16 in \cite{ChodoshLi}). We include here a proof to show how this can be obtained directly from Gromov's Theorem \ref{thmGBM}.
\begin{lemma}\label{diamEst}
    Let $(M^3,g)$ be a 3-manifold with $\scal_g\geq 6$ and let $\iota:\mathbb{S}^2\to M^n$ be a stable minimal immersion, then: 
    \[\diam(\mathbb{S}^2,\iota^*g)\leq\frac{2\pi}{3}.\]
\end{lemma}
\begin{proof}
    As in Gromov's first proof of Theorem \ref{ineqGro} in \cite{GromovM}, we will use Fischer-Colbrie--Schoen symmetrization.
    The stability of $\iota$ means that the Jacobi operator $J=-\Delta_{\iota^*g}-\tfrac{1}{2}\left(\scal_g-\scal_{\iota^*g}+|A|^2\right)$ is nonnegative.
    
    Let $f$ be the first eigenfunction of $J$ (which doesn't vanish since $J$ is a Schrödinger operator) and consider the metric $\tilde g=\iota^*g+f^2d\theta^2$ on $\mathbb{S}^2\times\mathbb{T}^1$. A classical computation shows that:
    \[\scal_{\tilde g}=\scal_{\iota^*g}-\frac{2\Delta_{\iota^*g}f}{f}\geq 6.\]
    
    Now, as in the proof of Theorem \ref{thmn}, we can apply Gromov's Theorem \ref{thmGBM} to $(\mathbb{S}^2\times\mathbb{T}^1,\tilde g)$ to show that $(\mathbb{S}^2,\iota^*g)$ has diameter at most $\tfrac{2\pi}{3}$.
\end{proof}

We are now ready to prove Theorem \ref{thm3}.
\begin{proof}[Proof of Theorem \ref{thm3}]
    First, we note that by going to the orientation cover if needed we can assume that $M^3$ is orientable.

    Let $(M^3,g)$ be a compact $3$-manifold with $\scal_g\geq 6$ and: 
    \begin{equation}\label{injassump}
        \inj_g>\tfrac{2\pi}{3}.
    \end{equation} 
    
    Since $(M^3,g)$ has positive scalar curvature, it follows from Perelman's solution to the Geometrization conjecture (\cite{Perelman}) that $M^3$ is diffeomorphic to a connected sum:
    \begin{equation*}\label{connsum}
        M^3\simeq P_1\# \cdots \# P_k
    \end{equation*}
    where the $P_i$ are either non simply connected spherical 3-manifolds or $\mathbb{S}^2\times\mathbb{T}^1$, note that we can have $k=0$ in which case $M^3$ is a 3-sphere and there is nothing to prove.
    
    If one of the $P_i$'s is $\mathbb{S}^2\times\mathbb{T}^1$, then by collapsing all the other factors we can build a non zero degree map $M^3\to \mathbb{S}^2\times\mathbb{T}^1$. Thus Theorem \ref{thmn} applies and $\inj_g\leq\tfrac{2\pi}{3}$. This contradicts \eqref{injassump}.
    
    Hence all the $P_i$'s are non simply connected spherical manifolds. 
    
    We will now prove that there is only one summand in the connected sum $\eqref{connsum}$. To see this note that if $k\geq 2$ then inside a neck $\mathbb{S}^2\times[-1,1]$ used to perform the connected sum one can find an embedded $2$-sphere which does not bound a $3$-ball. Now by Theorem 4.1 of \cite{HassScott}, we can minimize area among all such embedded spheres to obtain a stable minimal immersion $\iota:\mathbb{S}^2\to (M^3,g)$ (though it is not useful for our purpose, Hass and Scott actually prove that $\iota$ is either an embedding or a double cover of a 1-sided projective plane). 
    
     We can now apply Lemma \ref{diamEst} to show that
     $\diam(\mathbb{S}^2,\iota^*g)\leq\frac{2\pi}{3}$. Thus the image
     of $\iota:\mathbb{S}^2\to M^3$ is contained in a ball of radius $\tfrac{2\pi}{3}$. Since $\iota$ is homotopically non trivial this implies that the injectivity radius of $(M,g)$ is less than $\frac{2\pi}{3}$. This contradicts \eqref{injassump} and proves that $k=1$.
    
    We now have that $M^3$ is a spherical 3-manifold, and we can write it as a quotient $\mathbb{M}^3=\mathbb{S}^3/\Gamma$ for some finite subgroup $\Gamma\subset SO(4)$ of fixed point free isometries of the round $\mathbb{S}^3$.
    
    We will now show that $\Gamma$ has odd order. Assume to the contrary that $\Gamma$ has even order, then a classical exercise in group theory shows that $\Gamma$ contains an element $\sigma$ of order $2$. Since in $SO(4)$ we have $\sigma^2=I_4$, the eigenvalues of $\sigma$ are all equal to $\pm 1$. Since $\sigma$ acts on $\mathbb{S}^3$ without fixed points, $1$ cannot be an eigenvalue of $\sigma$. Hence $\sigma=-I_4$. 
    
    This shows that if $\Gamma$ has even order then $\{\pm I_4\}\subset \Gamma$. This gives rise to a covering $\Pi:\mathbb{RP}^3\to M^3$. Thus $\tilde g=\Pi^*g$ is a metric on $\mathbb{RP}^3$ with scalar curvature at least $6$. By \cite[Theorem 1]{BBEN}, this implies that the $1$-systole of $(\mathbb{RP}^3,\tilde g)$ is at most $\pi$, and thus that its injectivity radius is at most $\pi/2$. Since the injectivity radius can only increase by coverings, $\inj_g\leq\pi/2$ which contradicts \eqref{injassump}. Thus $\Gamma$ has odd order.
    
    All that needs to be shown now is that $M^3=\mathbb{S}^3/\Gamma$ is a lens space, this follows from the classification of spherical 3-manifolds. Inspecting section 7.5 of \cite{WolfSpaces} we see that all spherical 3-manifolds which are not lens spaces have fundamental group which contain binary dihedral, tetrahedral, octahedral or icosahedral groups as subgroups, thus all those non lens space spherical 3-manifolds will have even order fundamental group. Thus the only possible spherical 3-manifolds with odd order fundamental group are lens spaces.

\end{proof}
\begin{remark}
  It is a bit unsatisfying that the statement of \autoref{thm3} is
  not ``If $(M^3,g)$ has $\scal_g\geq 6$ and $\inj_g>2\pi/3$ then
  $M^3$ is homeomorphic to $\mathbb{S}^3$'' however my attempts at
  removing the odd order lens spaces from the statement have been
  inconclusive. One could try to use an index 1 minimal Clifford torus
  $T\subset L(p,q)$, whose area is less than $16\pi/3$ by
  \cite{MarquesNeves}. This gives an upper bound on the systole of $T$
  by Loewner's inequality however since  $T\subset L(p,q)$ is not
  injective at the level of $\pi_1$ this doesn't bound the systole of
  $L(p,q)$. Another approach could be to notice that if the
  injectivity radius of $L(p,q)$ is bigger than $2\pi/3$ then this
  gives a metric with diameter at least $4\pi/3$ on its universal
  cover, however \autoref{thm3diam} is not tight enough to be used there.
\end{remark}
\section{Injectivity radius bounds for 3-manifolds with large diameter}
\label{dim3diam}

In this section we prove Theorem \ref{thm3diam}.

We first state an elementary topological lemma, certainly well known, whose simple proof is
only included because of our inability to locate it somewhere in the literature:

\begin{lemma}\label{lemtop}
	Let $O^n$ be a smooth connected manifold and let $\Sigma^{n-1}$ be an hypersurface such that $O\backslash\Sigma$ has two non-compact connected components. Then $[\Sigma]\ne 0\in H_{n-1}(O)$.
\end{lemma}
\begin{proof}
	Since $\Sigma$ separates it is two-sided. Let $O_-$ and $O_+$ be the two components of $O\backslash \Sigma$. For an auxiliary complete metric $g$ on $O$, let $f:O\to\mathbb{R}$ be such that:
	\begin{itemize}
		\item in a small enough metric tubular neighborhood of $\Sigma$, $f(x)=\pm d_g(x,\Sigma)$ if $x\in O_\pm$.
		\item $f$ is smooth and proper.
	\end{itemize}
	This can be achieved by mollifying the signed distance function to $\Sigma$ which is proper since both $O_-$ and $O_+$ are not compact. We denote $\Sigma_t=\{f=t\}$.

	Let $\chi:\mathbb{R}\to\mathbb{R}$ be a smooth function supported in a small enough neighborhood of $0$ and such that $\int_\mathbb{R}\chi(t)dt=1$. Set $\eta=f^*(\chi(t)dt)$. Note that since a complete geodesic from $M_-$ to $M_+$ gives rise to a right inverse to $f$ and $\chi(t)dt$ is a generator of $H^1_c(\mathbb{R})$, $[\eta]\ne 0\in H^1_c(O)$.

		Let $\alpha$ be a smooth closed $n-1$-form, we have:
	\[\int_O\eta\wedge\alpha=\int_{\mathbb{R}}\chi(t)\left(\int_{\Sigma_t}\alpha\right)dt=\int_{\mathbb{R}}\chi(t)\left(\int_\Sigma\alpha\right)dt=\int_\Sigma\alpha.\]
since $[\Sigma_t]=[\Sigma]$ and $d\alpha=0$. Thus $[\eta]$ and $[\Sigma]$ are Poincaré dual to each other. 

Since $[\eta] \ne 0\in H^1_c(O)$, $[\Sigma]\ne 0\in H_{n-1}(O)$. 
\end{proof}

%
%
We can now move to the proof of Theorem \ref{thm3diam}.
\begin{proof}[Proof of Theorem \ref{thm3diam}.]
    Let $D=\diam(M^3,g)$ and let $p,q\in M^3$ be such that $d_g(p,q)=D$. For $r<\inj_g$, we set $\hat{M}=M\backslash\{p,q\}$ and $\tilde M=M\backslash\{B(p,r)\cup B(q,r)\}$.
    
    $\tilde M$ has two spherical boundary components $\partial_+\tilde M=S(p,r)$ and $\partial_-\tilde M=S(q,r)$ which satisfy:
    $d_g(\partial_-\tilde M,\partial_+\tilde M)=D-2r$.
    
    Assume that $D-2r>2\pi/3$, then $\delta=6-\frac{8\pi^2}{3(D-2r)^2}>0$ and we can apply Theorem \ref{thm-gro-mub} to $(\tilde M,g)$ to get a compact surface $\Sigma\subset \tilde M$ homologous to $S(p,r)$ in $\hat M$ and a positive function $f:\Sigma\to \mathbb{R}$ such that $\tilde g=g_{|\Sigma}+f^2d\theta^2$ has $\scal_{\tilde g}\geq\delta$ and $\Sigma\in [\partial_- M]$.
    
    We can now apply Theorem \ref{thmGBM} (after proper rescaling) to show that:
    \begin{equation}\label{diamdelta}
        \diam(\Sigma,g)\leq \frac{2\pi}{3\sqrt{\delta/6}}=\frac{2\pi}{3\sqrt{1-\frac{4\pi^2}{9(D-2r)^2}}}
    \end{equation}
     Since $\Sigma\subset\tilde M$, for any $x\in\Sigma$, $d_g(x,\{p,q\})>r$ and thus $B(x,r)\subset\hat M$. 

     Assume $r\geq \diam(\Sigma,g)$ then for any $x\in\Sigma$,
     $\Sigma$ is included in $B(x,r)$ which is a topological ball contained in
     $\hat M$ since $r$ is smaller than the injectivity radius.
    Hence $\Sigma$ is contractible in $\hat M$, which contradicts
    Lemma \ref{lemtop} with $O=\hat M$. Thus :
    \[r<\diam(\Sigma,g)\leq \frac{2\pi}{3\sqrt{1-\frac{4\pi^2}{9(D-2r)^2}}} \]
    Solving for $D$ we get the inequality:
    \begin{equation}
    D<2r+\frac{2\pi}{3\sqrt{1-\tfrac{4\pi^2}{9r^2}}}.\label{eq:diam-upper}
  \end{equation}
    We have proved that if $r<\inj_g$ and $D>2r+\frac{2\pi}{3}$ then
    \eqref{eq:diam-upper} holds.

    Hence by contraposition if $D\geq
    2r+\frac{2\pi}{3\sqrt{1-\tfrac{4\pi^2}{9r^2}}}$ then $r\geq\inj_g$
    or $D\leq 2r+\frac{2\pi}{3}$. Since for $r\in (2\pi/3,\pi)$,
    $\frac{2\pi}{3\sqrt{1-\tfrac{4\pi^2}{9r^2}}}>\frac{2\pi}{3}$, we
    get that $\inj_g\leq r$.
\end{proof}

\section{Injectivity radius of open 3-manifolds}
\label{sec:inject-radi-open}

In this section we prove Theorem \ref{thm3comp}. 
\begin{proof}[Proof of Theorem \ref{thm3comp}.]
  Fix a point $p$ in our open $(M^3,g)$ with $\scal_g\geq 6$. Consider
  a smooth proper $1$-Lipschitz function $\rho$ such that
  $|\rho-d(p,\cdot)|<1$. For $R>2+\tfrac{2\pi}{3}$ such that $R$ and
  $2R$ are regular values of $\rho$ we set:
  \[\tilde M=\rho^{-1}([R,2R]).\]
  Then $\partial \tilde M= \rho^{-1}({R})\cup \rho^{-1}({2R})$ and
  $[\rho^{-1}({R})]=[S(p,r)]$ in $H_2(M\backslash\{p\})$ for $r$ small
  enough since $S(p,r)$ and $\rho^{-1}({R})$ together bound $\rho^{-1}((-\inf,R))\backslash B(p,r)$. Moreover we have:
  \[d_g(\rho^{-1}({2R}), \rho^{-1}({R}))\geq R-2>2\pi/3.\]
  
  Thus we can apply Theorem \ref{thm-gro-mub} to $(\tilde M,g)$ with
  $\delta=6-\frac{8\pi^2}{3(R-2)^2}$ to get $\Sigma^2\subset\tilde M$ in the homology
  class of $\rho^{-1}(\{R\})$ (and thus of $S(p,r)$) and a function $f: \Sigma^2\to\mathbb{R}$ such
  that $g_{|\Sigma}+f^2d\theta^2$ has scalar curvature at least
  $\delta$. As in section \ref{dim3diam}, we get that:
  \[\diam(\Sigma^2,g)\leq\frac{2\pi}{3\sqrt{1-\frac{4\pi^2}{9(R-2)^2}}}.\]
  
  We now argue that $\inj_g\leq
  \diam(\Sigma^2,g)$. Let $x\in\Sigma^2\subset\tilde M$. Then
  $\Sigma^2\subset \overline B(x, \diam(\Sigma^2,g))\subset
  M\backslash\{p\}$ and, since $\Sigma^2$ is not contractible in
  $M\backslash\{p\}$ by Lemma \ref{lemtop}, $\overline B(x, \diam(\Sigma^2,g))$ is not
  contractible, hence:
  \[\inj_g\leq
    \diam(\Sigma^2,g)\leq\frac{2\pi}{3\sqrt{1-\frac{4\pi^2}{9(R-2)^2}}}.\]
  Picking $R$ as big as we want we get that:
  \[\inj_g\leq\frac{2\pi}{3}.\]
\end{proof}

\section{Injectivity radius of
  $\mathbb{S}^2\times\mathbb{T}^{n-3}\times\mathbb{R}$ and
  $\mathbb{S}^2\times\mathbb{T}^{n-4}\times\mathbb{R}^2$ }
\label{sec:inject-radi-mathbbs2}

Once again we start with a topological lemma:
\begin{lemma}\label{lemtopn}
  Let $M^n$ be a smooth $n$-manifold, $X^{n-1}$ be a compact $(n-1)$-manifold
  and $F:M\to X\times\mathbb{R}$ be smooth proper non zero-degree
  map. For any $p\in M$, set $F(p)=(x(p),r(p))\in X\times
  \mathbb{R}$. Let $r_0$ be a regular value of $r:M\to\mathbb{R}$,
  then for any $\Sigma\in [r^{-1}(r_0)]$, $x_{|\Sigma}:\Sigma\to X$
  has non-zero degree.
\end{lemma}
\begin{proof}
  Since $F$ has non-zero degree there exists $k\neq 0$ such for any
  compactly supported $n$-form $\omega$ on $X\times\mathbb{R}$:
  \begin{equation}
  \int_MF^*\omega=k\int_{X\times\mathbb{R}}\omega.\label{eq:deg}
\end{equation}
  Let $\Sigma_0=r^{-1}(r_0)$, $\Sigma_0$ is a compact smooth hypersurface. Since $r_0$ is a regular value of
  $r$, $F$ is a submersion from $M_\varepsilon=r^{-1}((r_0-\varepsilon,
  r_0+\varepsilon))$ for some positive $\varepsilon$, hence we can
  assume that $M_\varepsilon$ is diffeomorphic to $\Sigma_0\times (r_0-\varepsilon,
  r_0+\varepsilon)$ and that in these coordinates, $F$ can be written as:
  \begin{align*}
    F:M_\varepsilon\simeq \Sigma_0\times (r_0-\varepsilon,
    r_0+\varepsilon) &\to X\times\mathbb{R}\\
    (s,\rho)&\mapsto (x(s,\rho),\rho).
  \end{align*}
  
  Now let $\chi: (r_0-\varepsilon,
    r_0+\varepsilon)\to\mathbb{R}$ be a compactly
    supported function. Let $\xi$ be an $n-1$-form on $X$ which is not zero in
    $H^{n-1}(X)$. Then, then writing $t$ for the coordinate on the
    $\mathbb{R}$ factor we set $\omega=\chi(t)(\xi\wedge dt)$.  Then
    $\int_MF^*\omega=k\int_{X\times\mathbb{R}}\omega$ by \eqref{eq:deg}. We can compute:
    \[\int_{X\times\mathbb{R}}\omega=\int_{X\times\mathbb{R}}\chi(t)(\xi\wedge dt)=\int_\mathbb{R}\chi(t)
      dt\int_X\xi\]
    and:
    \[\int_MF^{*}\omega=\int_{M_\varepsilon}F^*(\chi(t)(\xi\wedge
      dt))=
      \int_{ (r_0-\varepsilon,
        r_0+\varepsilon)}\chi(\rho)\left(\int_{\Sigma_0}x^*\xi(s,\rho)\right)d\rho.\]
    
This shows that for any $\chi\in C^\infty_0((r_0-\varepsilon,
r_0+\varepsilon))$:
\[\int_{ (r_0-\varepsilon,
    r_0+\varepsilon)}\chi(t)\left(k\int_X\xi- \int_{\Sigma_0}x^*\xi(s,t)\right)dt=0\]
and thus for any $t\in (r_0-\varepsilon,
r_0+\varepsilon)$:
\[\int_{\Sigma_0}x^*\xi(s,t)=k\int_X\xi\neq 0.\]
For $t=r_0$, we get that $\int_{\Sigma_0}x^*\xi=k\int_X\xi\neq 0$ which
proves that $x|_{\Sigma_0}:\Sigma_0\to X$ has nonzero degree.

Now for any $\Sigma\in [\Sigma_0]$,
$\int_{\Sigma}x^*\xi=\int_{\Sigma_0}x^*\xi\neq 0$ hence  $x|_{\Sigma}:\Sigma\to X$
  has non-zero degree.
\end{proof}

We can now prove Theorem \ref{thmn} for $k=1$ or $2$.

\begin{proof}[Proof of Theorem \ref{thmn} for k=1,2.]
  For $k=1$, consider a proper non zerop degree map
  $F:(M^n,g)\to\mathbb{S}^2\times\mathbb{T}^{n-1}\times\mathbb{R}$ and let
  $r:M^n\to\mathbb{R}$ be the last coordinate of $f$. Let
  $L>0$ and set $M_L=r^{-1}([-L,L])$. Then $M_L$ is compact with
  boundary components $\partial_+M_{L}=r^{-1}(L)$ and
  $\partial_-M_{L}=r^{-1}(-L)$. Since $r$ is proper and $g$ is
  complete:
  \[d_g\left(\partial_-M_{L},\partial_+M_{L}\right)\xrightarrow{L\to +\infty}+\infty.\]

  By Sard's theorem, we can find an
  increasing sequence $(L_k)$ diverging to $+\infty$ such that $L_k$
  and $-L_k$ are regular values of $r$.

  We now apply Theorem \ref{thm-gro-mub} to $(M_{L_k},g)$ to find an hypersurface
  $\Sigma_k\in [\partial_+M_{L_k}]$ such that:
  \begin{itemize}
  \item By Lemma \ref{lemtopn}, $\pi\circ
    F_{|\Sigma_k}:\Sigma_k\to\mathbb{S}^2\times\mathbb{T}^{n-1}$ has
    non-zero degree where
    $\pi:\mathbb{S}^2\times\mathbb{T}^{n-1}\times\mathbb{R}\to\mathbb{S}^2\times\mathbb{T}^{n-1}$
    is the natural projection.
  \item there exists a smooth function $f_k:\Sigma\to (0,+\infty)$ such
    that $(\tilde M_k,\tilde
    g_k)=(\Sigma_k\times\mathbb{T}^1,g_{|\Sigma_k}+f_k^2d\theta^2)$ has
    $\scal_{\tilde g_k}\geq \delta_k$ where:
    \[\delta_k\xrightarrow{k\to+\infty} n(n-1).\]    
  \end{itemize}
  Note that $\tilde F=(\pi\circ
    F_{|\Sigma_k},\id_{\mathbb{T}^1}):\tilde M_k\to
    \mathbb{S}^2\times\mathbb{T}^{n}$ has non zero degree.
  
  We then apply Zhu's Theorem \ref{thmZhu} to $(\tilde M_k,\tilde g_k)$ to get a
  $2$-sphere $S_k\subset \tilde M_k$ and $n-2$ positive functions
  $f_{1,k},\dots,f_{n-2,k}:S_k\to\mathbb{R}$ such that:
  \begin{itemize}
  \item $\int_{S_k}\tilde F^*\sigma\neq 0$ where $[\sigma]\in
    H^2(\mathbb{S}^2\times\mathbb{T}^{n-2})$ is the fundamental class of
    the $\mathbb{S}^2$ factor.
  \item $(\hat M_k,\hat
    g_k)=(S_k\times\mathbb{T}^{n-2},\tilde{g}_{|S_k}+\sum_{i=1}^{n-2}f_{i,k}^2d\theta_i^2)$
    has $\scal_{\hat g_k}\geq \delta_k$.
  \end{itemize}
  Moreover it follows from Zhu's proof of Theorem \ref{thmZhu} that
  the construction of $S_k$ respects the $\mathbb{T}^1$ symmetry of
  $(\tilde M_k,\tilde g_k)$, which implies that $S_k\subset\Sigma_k$
  and $\tilde g_{|S_k}=(g_{\Sigma_k})_{|S_k}=g_{|S_k}$.

  By Theorem \ref{thmGBM} applied to $(\hat M_k,\hat g_k)$ , we have that the diameter of
  $(S_k,g_{|S_k})$ is at most $D_k$ where
  $D_k\xrightarrow{k\to+\infty}\frac{2\pi}{n}$. Since $S_k$ is not
  contractible in $M$, this implies that $\inj_g\leq D_k$. Letting $k$
  go to infinity we get that $\inj_g\leq \frac{2\pi}{n}$.

  We now turn to the case of $k=2$ and assume that $F:(M,g)\to
  \mathbb{S}^2\times\mathbb{T}^{n-4}\times\mathbb{R}^2$ is proper and
  has nonzero degree. Let $(x,y)$ denote cartesian coordinates of
  $\mathbb{R}^2$ and set:
  \begin{align*}
    r:M &\to\mathbb{R}\\
    p&\mapsto x(F(p))^2+y(F(p))^2.
  \end{align*}
  $r$ is smooth and proper. Let $\eta\in (0,1)$ be a regular value of
  $r$. For $L>1$ we set $M_L=r^{-1}((\eta,L))$,
  $\partial_-M_L=r^{-1}(\eta)$ and $\partial_+M_L=r^{-1}(L)$ . Since $r$ is proper
  and $g$ is complete:
  \[d_g\left(\partial_-M_{L},\partial_+M_{L}\right)\xrightarrow{L\to +\infty}+\infty.\]
  Moreover if one set
  $\dot{M}=F^{-1}\left(\mathbb{S}^2\times\mathbb{T}^{n-4}\times(\mathbb{R}^2\backslash\{0\})\right)$. Then
  $F_{|\dot M}$ has nonzero degree as a map valued in
  $\mathbb{S}^2\times\mathbb{T}^{n-4}\times(\mathbb{R}^2\backslash\{0\})\simeq
  \mathbb{S}^2\times\mathbb{T}^{n-3}\times\mathbb{R}$. We thus can
  argue as in the $k=1$ case.
\end{proof}

\appendix
\section{Non optimality of products for Gromov's diameter estimate}
\label{sec:non-optim-prod}
We will show here that one can deform the product metric on $\mathbb{S}^2\times\mathbb{T}^{n-2}$ in a way which increases the scalar curvature while keeping the diameter of the $\mathbb{S}^2$ factor, which shows that Theorem \ref{ineqGro} cannot be improved to show that product metrics are optimal.

More precisely, we will build metrics $g=g_s$ on $\mathbb{S}^2$ and functions $b=b_s:\mathbb{S}^2\to (0,\infty)$ for $s\geq 0$ such that:
\begin{enumerate}
    \item $g_0$ is the round metric on $\mathbb{S}^2$ and $f_0=1$.
    \item $\diam(\mathbb{S}^2,g_s)=\pi$.
    \item the metric $\tilde{g}=\tilde{g}_s=g_s+b_s^2dx^2$ on $\mathbb{S}^2\times\mathbb{T}^{n-2}$ (where $dx^2$ is a flat metric on $\mathbb{T}^{n-2}$) has $\scal_{\tilde g}>2$ for $s>0$ small enough.
\end{enumerate}

We will choose normal coordinates $(r,\phi)$ on $\mathbb{S}^2$ in
which the round metric is written as $g_0=dr^2+\sin^2(r)d\phi^2$ for
$r\in (0,\pi)$ and $\phi \in (-\pi,\pi]$. We will look at metrics of
the form $\tilde{g}=dr^2+a^2(r)d\phi^2+b(r)^2dx^2$. 

We will add to our functions $a$ and $b$ the next Fourier mode while ensuring smoothness of $\tilde{g}$ by setting:
\begin{itemize}
    \item $a(r)=\frac{\sin r+\alpha\sin 3r}{1+3\alpha}$,
    \item $b(r)=1+\beta \sin^2 r$,
\end{itemize}
for $\alpha,\beta$ to be determined later. Note that for $\alpha\in
(-\tfrac{1}{3},1)$ and $\beta\in(-1,+\infty)$,
$\tilde{g}_s=dr^2+a^2(r)d\theta^2+b(r)^2dx^2$ is a smooth metric on
$\mathbb{S}^2\times\mathbb{T}^{n-2}$ and that since $(\mathbb{S}^2,dr^2+a^2(r)d\theta^2)$ has diameter at least $\pi$. The scalar curvature of $\tilde g$ is given by:
\begin{align*}
\scal_{\tilde{g}}=&-2\left(\frac{a'}{a}+(n-2)\frac{b'}{b}\right)'\\
  &-\left(\left(\frac{a'}{a}\right)^2
  +(n-2)\left(\frac{b'}{b}\right)^2\right)
-\left(\frac{a'}{a}+(n-2)\frac{b'}{b}\right)^2.
\end{align*}
This can be computed for instance with Cartan's method of moving
coframes applied to the orthonormal coframe $(dr,a(r)d\phi,b(r)dx_3,\dots,b(r)dx_n)$.

With the help of a computer algebra system\footnote{this computation
  and the subsequent computations and plots have been
  carried on SageMath 10.4, the reader can find these in the sage
  notebook \cite{notebook} which can be downloaded at \url{https://doi.org/10.5281/zenodo.13933633}.}, we can plug our expressions $a$ and $b$ inside the expression above for the scalar curvature. We get a lengthy formula depending on $r$, $\alpha$, $\beta$ and $n$. 

In Figure \ref{fig:plot-scal}, we plotted the scalar curvature of
$\tilde g$ for various choices of our parameters $\alpha$, $\beta$ and
$n$, see also \cite{notebook}. From this plot we can guess that $r\mapsto \scal_{\tilde g}$
attains its minimum at either $r=0$ or $r=\pi/2$. We can then use our
CAS to compute those values:
\begin{itemize}
    \item $\scal_{\tilde g}(0)=\frac{2-4\beta(n-2)(6\alpha+2)+54\alpha}{1+3\alpha}$
    \item $\scal_{\tilde g}(\pi/2)=\frac{2(1-9\alpha-3\beta-5\alpha\beta+2n(1-\alpha)\beta)}{(1-\alpha)(1+\beta)}$
\end{itemize}
\begin{figure}
    \centering
    \includegraphics[scale=0.75]{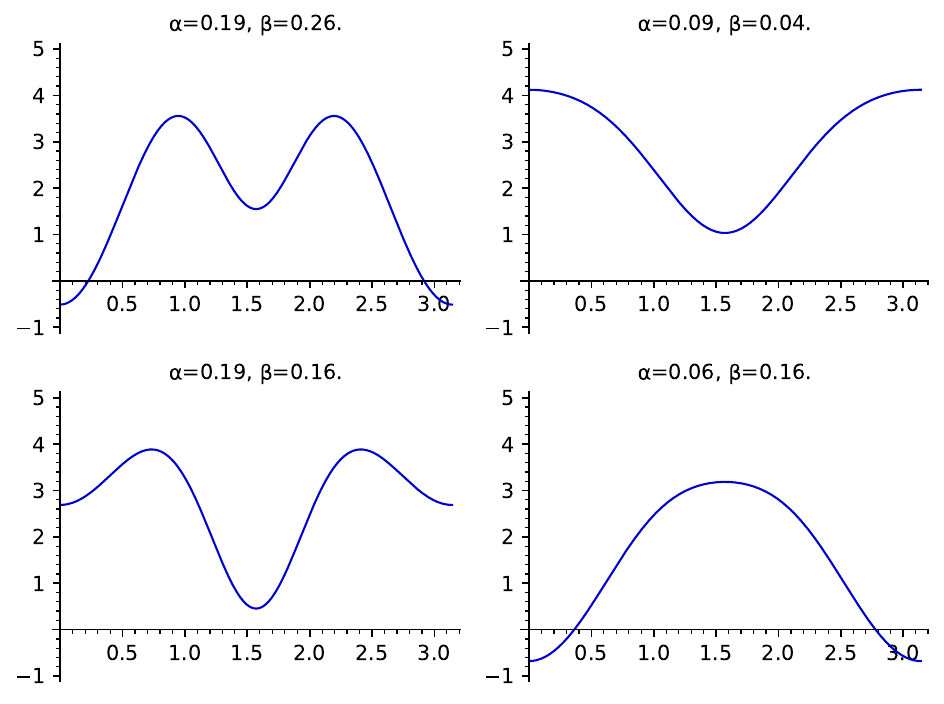}
    \caption{Example plots of the scalar curvature of $\tilde{g}$ as a
      function of $r$ for
      various choices of the $\alpha$ and $\beta$ parameters, for
      $n=6$, see \cite{notebook} for interactive plots with adjustable
      $(\alpha,\beta,n)$ values.}
    \label{fig:plot-scal}
\end{figure}
  
To find candidates for counter examples, we plotted for various values of $n$ the regions $\mathcal{R}_0=\{\scal_{\tilde g}(\pi/2)\geq 2\}$ and $\mathcal{R}_{\pi/2}=\{\scal_{\tilde g}(0)\geq 2\}$ in the $(\alpha,\beta)$ plane. Luckily enough the two regions seem to always intersect in what looks like a tiny convex region of the plane as seen on Figure \ref{fig:plot}.
\begin{figure}
    \centering
    \includegraphics[scale=0.75]{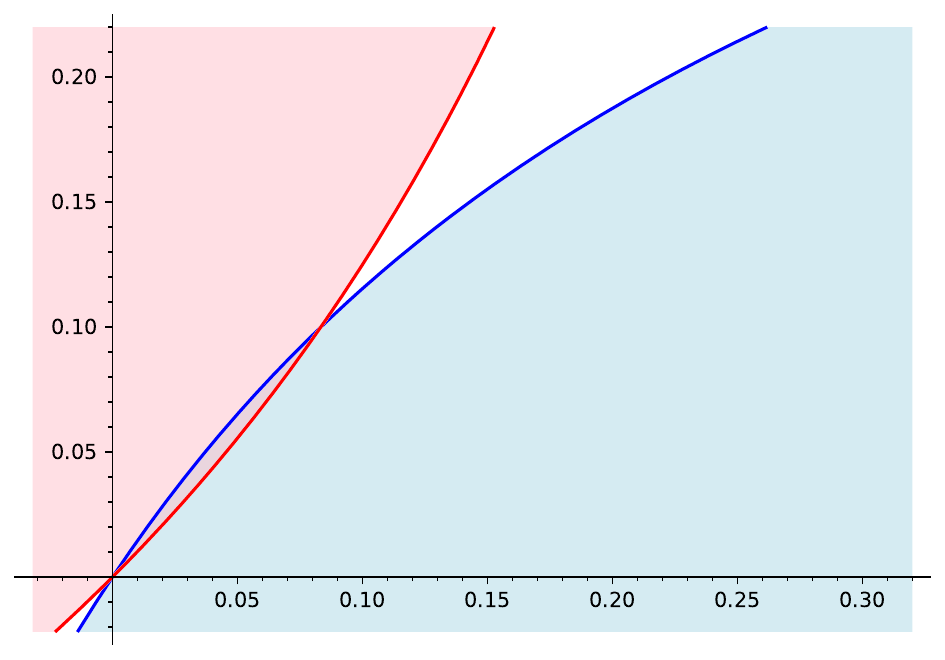}
    \caption{Plot of the $\mathcal{R}_0$ (blue) and
      $\mathcal{R}_{\pi/2}$ (pink) regions in the
      $(\alpha,\beta)$-plane for $n=4$, see \cite{notebook} for
      other values of $n$.}
    \label{fig:plot}
\end{figure}

In \cite{notebook} we solved the system:
\[\begin{cases}
\scal_{\tilde g}(\pi/2)= 2\\
\scal_{\tilde g}(0)= 2\\
\end{cases}
\]
for $\alpha$ and $\beta$, and obtained two solutions $(0,0)$ and $(\tfrac{n-2}{3(3n-2)},\tfrac{1}{2(n-1)})$. Since we have observed that our counterexample set $\mathcal{R}_{\pi/2}\cap\mathcal{R}_{0}$ should be convex in the $(\alpha,\beta)$ plane, the whole segment between these two points should consist of counterexamples. Hence we set
\[(\alpha,\beta)=s(\tfrac{n-2}{3(3n-2)},\tfrac{1}{2(n-1)})\]

What remains to be shown is that at least for $s$ small enough, we
actually have $\scal_{\tilde g_s}>2$.

In order to see this we can compute \cite{notebook}
\[\left .\frac{\partial\scal_{\tilde g_s}}{\partial s}\right |_{s=0}=\frac{(n-2)((7n-10)+(5n-14)\cos 2r)}{3(n-1)(3n-2)}\geq\frac{(n-2)(2n+4)}{3(n-1)(3n-2)}\]
which is strictly positive since $n\geq 3$. Hence for $s>0$ small
enough, $\scal_{\tilde g_s}>\scal_{\tilde g_0}=2$.

\begin{remark}
  While this shows that products are not optimal (even locally or infinetesimaly) for the diameter
  estimate given by \autoref{thmGBM},
  this doesn't rule out optimality of products for the injectivity
  radius estimate of \autoref{thmn}. This follows from the fact that
  while the length of the meridians of $\mathbb{S}^2$ remains fixed at
  $\pi$ for our counterexamples, the equator is shrunk when $s>0$
  which lowers the injectivity radius too fast with respect to the increase in
  minimal scalar curvature as is shown in \cite{notebook}.
\end{remark}

\bibliographystyle{alpha}
\bibliography{psc}

\begin{thebibliography}{BBEN10}

\bibitem[BBEN10]{BBEN}
H.~Bray, S.~Brendle, M.~Eichmair, and A.~Neves.
\newblock Area-minimizing projective planes in 3-manifolds.
\newblock {\em Commun. Pure Appl. Math.}, 63(9):1237--1247, 2010.

\bibitem[Cec20]{Cecchini}
Simone Cecchini.
\newblock A long neck principle for {Riemannian} spin manifolds with positive
  scalar curvature.
\newblock {\em Geom. Funct. Anal.}, 30(5):1183--1223, 2020.

\bibitem[CL23]{ChodoshLi}
Otis Chodosh and Chao Li.
\newblock Generalized soap bubbles and the topology of manifolds with positive
  scalar curvature, 2023.

\bibitem[Gre63]{GreenWieder}
L.~W. Green.
\newblock Auf-{Wiedersehens}-{Fl{\"a}chen}.
\newblock {\em Ann. Math. (2)}, 78:289--299, 1963.

\bibitem[Gro18]{GromovM}
Misha Gromov.
\newblock Metric inequalities with scalar curvature.
\newblock {\em Geom. Funct. Anal.}, 28(3):645--726, 2018.

\bibitem[Gro23a]{Gromov4}
Misha Gromov.
\newblock Four lectures on scalar curvature.
\newblock In {\em Perspectives in scalar curvature. In 2 volumes}, pages
  1--514. Singapore: World Scientific, 2023.

\bibitem[Gro23b]{gromov2023scalar}
Misha Gromov.
\newblock Scalar curvature, injectivity radius and immersions with small second
  fundamental forms, 2023.

\bibitem[HS88]{HassScott}
Joel Hass and Peter Scott.
\newblock The existence of least area surfaces in 3-manifolds.
\newblock {\em Trans. Am. Math. Soc.}, 310(1):87--114, 1988.

\bibitem[HXZ25]{HuXuZhang}
Qixuan Hu, Guoyi Xu, and Shuai Zhang.
\newblock The sharp diameter bound of stable minimal surfaces.
\newblock {\em The Journal of Geometric Analysis}, 35(7):197, 2025.

\bibitem[MN12]{MarquesNeves}
Fernando~C. Marques and Andr{\'e} Neves.
\newblock {Rigidity of min-max minimal spheres in three-manifolds}.
\newblock {\em Duke Mathematical Journal}, 161(14):2725 -- 2752, 2012.

\bibitem[Per03]{Perelman}
Grisha Perelman.
\newblock Ricci flow with surgery on three-manifolds.
\newblock {\em arXiv e-print service}, 2003:22, 2003.
\newblock Id/No 0303109.

\bibitem[Ric24]{notebook}
Thomas Richard.
\newblock {Increasing the scalar curvature of
  $\mathbb{S}^2\times\mathbb{T}^{n-2}$ while keeping the diameter of the
  $\mathbb{S}^2$ factor fixed.}
\newblock \url{https://doi.org/10.5281/zenodo.13933633}, October 2024.
\newblock SageMath 10.4 notebook.

\bibitem[TW19]{Tuschmann}
Wilderich Tuschmann and Michael Wiemeler.
\newblock Smooth stability and sphere theorems for manifolds and {Einstein}
  manifolds with positive scalar curvature.
\newblock {\em Commun. Anal. Geom.}, 27(2):491--509, 2019.

\bibitem[Wol11]{WolfSpaces}
Joseph~A. Wolf.
\newblock {\em Spaces of constant curvature}.
\newblock Providence, RI: AMS Chelsea Publishing, 6th ed. edition, 2011.

\bibitem[Zhu20]{Zhu}
Jintian Zhu.
\newblock Rigidity of area-minimizing {{\(2\)}}-spheres in {{\(n\)}}-manifolds
  with positive scalar curvature.
\newblock {\em Proc. Am. Math. Soc.}, 148(8):3479--3489, 2020.

\end{thebibliography}
\end{document}